\documentclass{amsart}
  \usepackage{amscd,amssymb,epsfig, epsf}
   \usepackage{epic,eepic}
 \usepackage{epstopdf}

    \usepackage[all]{xy}
    \SelectTips{cm}{}
    \allowdisplaybreaks

    \numberwithin{equation}{subsection}

    \newtheorem{propo}{Proposition}[section]
    \newtheorem{corol}[propo]{Corollary}
    \newtheorem{theor}[propo]{Theorem}
    \newtheorem{lemma}[propo]{Lemma}

    \theoremstyle{definition}

    \theoremstyle{remark}

\newcommand{\ZZ}{\mathbb{Z}}

\newcommand{\RR}{\mathbb{R}}
\newcommand{\Z}{\mathbb{Z}}

\newcommand{\Hom}{\operatorname{Hom}}
\newcommand{\Ker}{\operatorname{Ker}}

\newcommand{\Int}{\operatorname{Int}}

\newcommand{\card}{\operatorname{card}}

\newcommand{\id}{\operatorname{id}}

\let\oldmarginpar\marginpar
\renewcommand\marginpar[1]{\oldmarginpar{\footnotesize #1}}

\begin{document}

    \title[{Additive posets and CW-complexes}]{Additive posets,  CW-complexes, and graphs}

    \author[Vladimir Turaev]{Vladimir Turaev}
    \address{
   Department of Mathematics \newline
    \indent  Indiana University \newline
    \indent Bloomington IN47405, USA\newline
    \indent $\mathtt{vturaev@yahoo.com}$}

                     \begin{abstract}  We introduce  and study   additive posets.
                     We show that the top homology group (with coefficients in $\Z/2\Z$) of a  finite dimensional CW-complex carries a   structure of an additive poset invariant under subdivisions. Applications    to  CW-complexes and graphs are discussed.

 \end{abstract}\footnote{AMS Subject
                    Classification:  05C10, 05C38, 06A11,  57Q05}
                     \maketitle


   \section {Introduction}

 This paper started with the
following elementary question: which 1-dimensional homology classes   of   graphs
are realizable by embedded circles? For example, if a graph~$X$ is a wedge of
  circles, then these circles are embedded in~$X$ and represent
generators of  the homology group $H=H_1(X; \Z/2\Z)$. All other elements of~$H$ cannot
be realized by embedded circles. An analysis of this phenomenon leads us to the
formalism of additive posets. The notion of an additive poset  introduced here combines the   concepts of an abelian group and a  poset in a way different from partially ordered groups.
A basic example of an additive poset is provided by the power set of a  set with the   order determined by the inclusion   and the addition determined by the symmetric difference.     We develop the theory of additive posets and, in particular, define their distinguished elements     called atoms and tiles.

 Additive posets   naturally arise    in the study of   graphs and, more generally,   CW-complexes.
  For a   CW-complex~$X$ of finite dimension $n\geq 0$,  each element of  the  top homology group  $H=H_n(X; \Z/2\Z)$
  is uniquely represented by an $n$-cycle, that is by a   finite collection of  $n$-cells of $X$ such that every $(n-1)$-cell  of $X$ is incident to an even number of $n$-cells in this collection  (counting with multiplicities).  For $a,b\in H $, we set $a\leq b$ if the collection of  $n$-cells representing~$a$ is contained in the collection of  $n$-cells representing~$b$.
  This  defines a partial order  in~$H$ and   turns~$H $  into an additive poset  preserved  under subdivisions of~$X$.  We call~$H$  the homology poset of~$X$.    We prove that for   $n \geq 2$, each finite additive poset is realizable as  the homology poset  of an   $n$-dimensional CW-complex (Theorem \ref{realizadpos}).    For $n=1$, i.e., for graphs, a similar claim does not hold: some  finite additive posets  are not realizable  as   the homology posets of  graphs (Theorem \ref{trealgraphs}).  

As a geometric application, we    estimate from below   the number of $n$-cells of an $n$-dimensional CW-complex~$X$    in terms of the partial order in   $H=H_n(X; \Z/2\Z)$.
This improves the standard estimate (the number of $n$-cells) $ \geq \dim_{\Z/2\Z} H$.  We also consider the problem of realization of elements of~$H $ by embeddings of  closed   $n$-manifolds    into~$X$.
 We show that  only  tiles of~$H $   can be realized by  such embeddings  (Theorem \ref{repr}). Only atoms  of $H$  can be realized by embeddings of closed
  connected $n$-manifolds  into~$X$.
For   $n=1$,  we have a more precise statement: a 1-dimensional homology class  of a   graph  can be  realized by an embedded circle   if and only if this homology class  is an atom (Theorem \ref{reprgraph}).  


 \section {Additive posets and their morphisms}

\subsection{Additive posets}\label{Definition and  properties} Recall that a (non-strict) partial order in a set~$P$ is a binary relation  $\leq$  over~$P$ such that  for all $a, b, c \in P $, we have:  $a \leq  a$ (reflexivity); if $a \leq   b$ and $b \leq  a$, then $a=b$ (antisymmetry); if $a \leq   b$ and $b \leq   c$, then $a \leq   c$ (transitivity). A set endowed with a partial order is called a partially ordered set or  a poset.

By an   \emph{additive  poset} we  mean a pair (an abelian group $A$,  a partial order $\leq$   in~$A$) such that for any   $a,b,c\in A  $, the following two conditions are met:

$(\ast)$           if $ b \leq a$ and $c \leq a$, then $b+c \leq a$;

$(\ast \ast)$           if $ a \leq b$ and $a \leq c$, then $a  \leq a+b+ c$.

\noindent
We   will denote the zero element of~$A$ by $0$ and will usually denote an additive poset  $(A, \leq)$  simply by $A$.

 \begin{lemma}\label{lemma1}   For any element~$a $ of an additive poset,  $a+a=0$ and $0 \leq a$.\end{lemma}

\begin{proof}
 By the definition of a partial order, $a \leq a$.   Condition $(\ast)$   implies that  $a+a \leq a$. A second application of  $(\ast)$   yields   $a+a+a \leq a$.
Also, an  application of Condition $(\ast \ast)$ to the relation $a\leq a$ yields  $ a\leq a+a+a$.
Therefore $a+a+a=a$. Consequently,  $a+a=0$ and   $0=a+a \leq a$. \end{proof}

 Lemma \ref{lemma1} shows  that all nonzero elements of an additive poset  have  order two. This allows us  to treat  additive posets as vector spaces over
the  field   $ \Z/2\Z=\{0,1\}$. Lemma \ref{lemma1} implies  that  the zero vector   is the (unique)  least element  of an  additive poset. We will often use the following corollary of this  fact: if a vector~$a $ in an additive poset satisfies  $a \leq 0$, then  $a=0$.

The  partial order in an additive poset $A$ restricted to any subgroup of~$A$ turns the latter
 into an additive poset   called an \emph{additive subposet} of $A$.

\subsection{Examples}\label{exam16}  1.   Any $\Z/2\Z$-vector space  $A$  carries the \emph{trivial partial order} $\leq_t $ defined by $0\leq_t a \leq_t a$ for all $a\in A$ (and no other relations). The pair $(A, \leq_t)$ is an additive poset.


2. Each set $I$ gives rise to an additive poset $2^I$ whose elements are arbitrary subsets of $I$.  The partial order in $2^I$ is determined by the inclusion: $X \leq Y$   if $X \subset Y$ for $X, Y \in 2^I$.
Addition  in $2^I$ is   the symmetric difference:  $X+Y =(X \backslash Y) \cup (Y\backslash X)$ for   $X, Y \in 2^I$. The zero element of $2^I$ is  the empty set. All properties of an additive poset are straightforward. We call   $2^I$ the  \emph{additive powerset}
of~$I$.

The  additive powerset $2^I$ can be described  in terms of maps $I \to\Z/2\Z$. Such maps bijectively correspond to subsets of $I$ by assigning to every set $X \subset I$ its characteristic map  $I \to\Z/2\Z$ carrying all elements of $X$ to $1$ and all elements of $I\backslash X$ to $0$.  Addition in  $2^I$  corresponds in this language to the addition of functions.  The partial order in $2^I$ is formulated in this language   as follows: two maps $a,b : I \to\Z/2\Z$ satisfy $a\leq b$   if and only if
  $a(i) \leq_t  b(i)$ for all $i\in I$, where $\leq_t$ is the trivial partial order    in  $ \Z/2\Z$.

 3. Finite subsets of a set $I$ form an additive subposet of $2^I$ denoted $2^I_f$.  It can be  equivalently defined  in terms of the $\Z/2\Z$-vector space
 $A=\oplus_{i\in I} \, (\Z/2\Z) \,i$.
 We   identify $2^I_f=A$ by assigning  to each finite set $J\subset I$ the vector $\sum_{i\in J} i\in A$.
The    addition and partial order in $2^I_f$ correspond    to the addition in~$A$ and the following partial order in~$A$: a vector     $ \sum_{i\in I} a_i i$ is smaller than or equal to a vector     $ \sum_{i\in I} b_i i$ (with $a_i, b_i\in \Z/2\Z$) if and only if  $a_i\leq_t b_i$ for all~$i$.
If   $I$ is a finite set, then $2^I=2^I_f$.


4. Given a set $I$, the finite subsets of~$I$ with an even number of elements  form an
additive subposet of $2^I_f$  denoted by  $2^I_{ev}$. 

5. Let $R$ be a Boolean ring that is a ring such that $a^2=a$ for all $a\in R$. It is   known that such a ring   is commutative and $a+a=0$ for all $a\in R$. The canonical partial order   in $R$ is defined by $a \leq b$ if $a=ab$. It is easy to check that the underlying additive group of~$R$ with this partial order   is an additive poset.  

6. For  additive posets $A$ and $B$, consider the direct sum of their underlying abelian groups with the  cartesian partial order: $(a,b) \leq (a', b')$ if $a\leq a'$ and $b \leq b'$ where $a,a'\in A$ and $b,b'\in B$. This gives an additive poset denoted $A \oplus B$.

\subsection{Morphisms of additive posets}  A  \emph{morphism} of additive posets $A\to B$ is a group homomorphism $\varphi : A \to B$ which is \emph{order-preserving} in the sense
 that for all $a, b\in A$ satisfying $a \leq b$, we have  $\varphi(a) \leq \varphi(b)$.  Additive posets and their morphisms   form a category with the obvious  composition of morphisms. An isomorphism in this category is a bijective map between additive posets which is both a group isomorphism and a poset isomorphism.

For example, a   map $g $ from a  set  $I$ to a set  $J$ induces a morphism of  additive posets $g^\ast: 2^J \to 2^I$    carrying  each subset of~$ J$ to its pre-image under~$g$. If $g$ is  injective, then it induces     a morphism  of additive posets
 $g_\ast : 2^I \to 2^J $ carrying each subset of~$ I$ to its  image under~$g$.
 If $g$ is a bijection, then $g^\ast$ and $g_\ast$ are mutually inverse isomorphisms of additive posets.

\subsection{Remarks} 1. The partial order in  a nonzero additive poset $A$ is not translation invariant. Indeed, for any nonzero $a\in A$ we have $0\leq a$. If the order is
translation invariant, then   $a=a+0 \leq a+a=0$, i.e., $a=0$, a contradiction.

2. Three elements  $a,b,c$ of   an additive poset    satisfy the relations $a \leq b$, $a\leq c$, and $a  \leq b+ c$ if and only if $a=0$.
The \lq\lq if'' part is obvious. The \lq\lq only  if'' part: by   $(\ast \ast)$, the first two relations imply that  $a  \leq a+b+ c$. If  $a  \leq b+ c$, then another application of $(\ast \ast)$ gives $$a\leq a+(a+b+c)+(b+c)=(a+a)+(b+b)+(c+c)=0.$$
Consequently,  $a=0$.

3. Given an  additive poset $A$, consider the   partial binary operation $\setminus$ on~$A$ defined by $a \setminus b=a+b$ for $a,b\in A$ such that $b \leq a$.  This   is a difference operation   in the sense of  \cite{KC}. If $A$ has a greatest element, then
 the underlying poset of $A$   with the operation $\setminus$ is a difference poset in the sense of \cite{NP, KC}.

\section{A construction of additive posets}\label{A construction of additive posets}

\subsection{Construction}\label{A construction of additive posets1} We show how to  construct   additive posets from families of linear functionals. We start with a lemma.

\begin{lemma}\label{presdold}  Let $A$ be a $\Z/2\Z$-vector space    and let $(B, \leq)$ be an additive poset. Any set  $S $ of linear maps  $ A \to B$   determines a binary relation $\preceq $ in  $A$ as follows:  $a \preceq b$ for $a,b \in A$ if (and only if) $s(a) \leq s(b)$ for all $s\in S$.  If  $\cap_{s\in S} \Ker s=0$, then  $(A, \preceq)$ is an additive poset. \end{lemma}

\begin{proof} Reflexivity and transitivity of   $\preceq $   follow from the corresponding properties of~$\leq$. To prove the antisymmetry of   $\preceq $, consider any $a,b \in A$ such that $a\preceq b$ and $b \preceq a$. Then $s(a) \leq s(b)$ and $s(b) \leq s(a)$ for all $s\in S$. Since $\leq$ is a partial order, $s(a) = s(b)$ for all $s\in S$. Therefore $a-b \in \cap_{s\in S} \Ker s=0 $ so that $a=b$.

We   verify  Condition $(\ast)$   in the definition of an additive poset. If  $a,b,c \in A  $ satisfy
$ b \preceq a$ and $c \preceq a$, then for any $s\in S$, we have $s(b) \leq s(a)$ and $s(c) \leq s(a)$. Since $s$ is an additive map and   $ (B , \leq)$   is an additive poset,
$$s(b+c)=s(b)+s(c) \leq s(a).$$   Consequently,  $b+c \preceq a$. Condition $(\ast \ast)$ is verified similarly: if  $a,b,c \in A  $ satisfy    $ a \preceq b$ and $a \preceq c$, then for all $s\in S$, we have $s(a) \leq s(b)$ and $s(a) \leq s(c)$. Since~$s$ is an additive map and  $ (B , \leq)$ is an additive poset,
$$s(a) \leq s(a)+s(b)+s(c) = s(a+b+c).$$  Consequently,  $a \preceq a+b+c $. \end{proof}

 Applying Lemma \ref{presdold} to  $B=\Z/2\Z$ with the trivial partial order $\leq_t$,  we obtain  the following theorem.

\begin{theor}\label{presd}  Let $A$ be a $\Z/2\Z$-vector space and $  A^*= \Hom(A, \Z/2\Z)$. Any  set $S \subset A^*$      determines a binary relation $\preceq $ in  $A$ as follows:  $a \preceq b$ for $a,b \in A$ if (and only if) $s(a) \leq_t s(b)$ for all $s\in S$. If  $\cap_{s\in S} \Ker s=0$, then the pair $(A, \preceq)$ is an additive poset. \end{theor}


To ensure  the condition  $\cap_{s\in S} \Ker s=0$ in Theorem~\ref{presd}, it suffices to  require  that~$S$ generates $A^*$ as a vector space.

\subsection{Examples}\label{exam22} In the following examples, $A$ is a $\Z/2\Z$-vector space.

 1. For  $S=A^*$,  Theorem \ref{presd} yields the trivial partial order  in~$A$. Indeed, for any distinct nonzero vectors $a,b \in A$, there is a homomorphism $s: A \to \Z/2\Z$ such that $s(a)=1$ and $s(b)=0$.
This shows that $a \not\leq b$ and leaves  only   the   relations $0\leq a \leq a$ for all $a\in A$.

2. For a    nonzero vector $a_0 \in A$, consider the set $S\subset A^*$ consisting of all homomorphisms $  A \to \Z/2\Z$ which   carry  $a_0$ to~$1$.   Applying Theorem~\ref{presd}   we obtain a partial order in~$A$ which  turns $A$ into an additive poset. This partial order may be directly defined by    $0\leq a \leq  a \leq a_0$ for all $a\in A$.

3.   For a   $\Z/2\Z$-vector subspace~$H$  of~$A$ and a nonzero vector $a_0 \in H$, let $S\subset A^*$ consist  of all homomorphisms $  A \to \Z/2\Z$ which either carry  $a_0$ to~$1$ or carry  $H$ to $0$.   Applying Theorem~\ref{presd}   we obtain a partial order in~$A$ which  turns~$A$ into an additive poset. This partial order may be directly defined by    $0\leq a \leq  a$ for all $a\in A$ and $a\leq a_0$ for all $a\in H$. For $H=A$, we obtain the   previous example.



\section{Tails and independence}\label{The independence relation and atoms}



\subsection{Tails} For any element~$a$ of an additive poset~$A$, the set
 $A_a= \{b\in A \, \vert \, b \leq a\}$  is  a subgroup of $A$ as follows   from the relation $0\leq a$ and Condition $(\ast)$  in the definition of an additive poset.  The subgroup   $A_a$ is  called the \emph{tail} of~$a$.  We have  $a\in  A_a$ and $A_0=\{0\}$.

 We   say that a subgroup~$B$ of~$A$ is \emph{full} if $A_b \subset B$ for all $b\in B$.
 Clearly, the tail of any element of~$A$ is a full subgroup of~$A$.

\subsection{Independence} We say that elements $a$, $b$ of an additive poset~$A$ are      \emph{independent} (from each other)  if $a \leq a+b$. Since  $a+b \leq a+b$,    Condition $(\ast)$   in Section~\ref{Definition and  properties}   implies that  for independent $a$, $b$, we have $$b=a+(a+b) \leq a+b=b+a.$$  Thus,   the relation of independence is symmetric.  Independent  nonzero vectors  $a,b \in A$ are   incomparable in the sense that neither $a \leq b$ nor $b   \leq a$. Indeed, if, for example, $a\leq b$, then applying   $(\ast \ast)$ to the relations $ a \leq b $ and $a \leq a+b$, we obtain that $a \leq 0$, which contradicts the assumption $a\neq 0$.

Example:     elements of the  additive powerset $2^I$ of a  set~$I$ are independent if and  only if they are disjoint as subsets of $I$. The same holds for  $2^I_f$ and $2^I_{ev}$.

\begin{lemma}\label{lemma4}  Let     $A^a   $ be the set of all elements of an additive poset $ A$  independent from $a\in A$. Then  $A^a $ is a full subgroup of~$A$ and $A^a\cap A_a=\{0\}$. Also, $A^b\subset A^a$ for all $a\leq b$.  \end{lemma}

\begin{proof}   The definition of independence implies that   $0\in A^a$.
For any $b,c\in A^a$, we have  $a \leq a+b$ and $a \leq a+c$.
By Condition $(\ast \ast)$ from Section~\ref{Definition and  properties},
$$ a \leq a+(a+b) + ( a+c)=a+b+c.$$
Thus, $b+c\in A^a$, i.e., $A^a $ is a subgroup of $A$. The equality $A^a\cap A_a=\{0\}$ holds because all nonzero elements of $A^a$ are incomparable with~$a$.

To prove that $A^a$ is full, we should check that    $A_b\subset A^a$ for all $b\in A^a$. Pick any  $c\in A_b $. Then  $ c\leq  b\leq a+b$, and therefore $c \leq a+b$. Applying Condition $(\ast \ast)$  to the relations  $c\leq b$ and  $c \leq a+b$,
we obtain that $$c \leq c+b+ (a+b)=a+c.$$ Thus, $c\in A^a$.

To prove that $A^b\subset A^a$ for  $a\leq b$, pick any $c\in A^b$. Then  $a\leq b \leq b+c$. Thus, both $a \leq b$ and  $a\leq b+c$. Applying  $(\ast \ast)$, we get $a\leq a+c$ so that $c\in A^a$. \end{proof}

\begin{lemma}\label{lemma--4} Any   pairwise independent nonzero vectors in an additive poset   are linearly independent (over $\Z/2\Z$).
 \end{lemma}

\begin{proof} If the   vectors  in question  are not linearly independent, then some of them, say, $\{b_1,..., b_n\}$   satisfy $b_1+ \cdots +b_n=0$ where $n\geq 1$. For $n=1$, this contradicts the assumption $b_1 \neq 0$. Suppose that $n \geq 2$. Since the vectors $b_2,..., b_n$ are independent from~$b_1$,   their sum   is also independent   from~$b_1$, see    Lemma~\ref{lemma4}. Then
  $$b_1 \leq b_1+ (b_2+ \cdots +b_n)=0.$$ Consequently, $b_1=0$ which contradicts the assumptions of the lemma.
   \end{proof}

    For  elements $a,b$ of
a poset $P$, one writes
$a< b$ if $a \leq b$ and $a \neq b$.  A \emph{chain  of length $n \geq 1$ in $P$}   is a   sequence    $a_0,a_1,..., a_n \in P $ such that $a_0<a_1 < \cdots <a_n$.

    \begin{lemma}\label{lemma--5}
     For any
     element~$a $ of an additive poset~$A$ and any integer $n \geq 1$, there is  a   bijective correspondence between chains of length~$n$  in $A$ starting with~$a$  and
  sequences of~$n$ pairwise independent nonzero vectors in~$A^a$. The correspondence carries   a  chain   \begin{equation}\label{chain}  a=a_0<a_1<a_2 < \cdots <a_n \end{equation}
  into the sequence $b_1,..., b_n$ where $b_i=a_{i-1} +a_{i} $ for $i=1,..., n$. The inverse correspondence
  carries a sequence $b_1,..., b_n \in A$ into the chain \eqref{chain} defined by   $a_i =a+ b_1 + \ldots   + b_i $ for     $i=1,..., n$. \end{lemma}

\begin{proof}  Consider a  chain  \eqref{chain} in $A$ and set $b_i=a_{i-1} +a_{i}  \in A$ for   $i=1, \ldots, n$. We claim that $b_1,..., b_n$ are pairwise independent nonzero vectors of $A^a$.
  That $b_i \neq 0$ for all~$i$ follows from the assumption $a_{i-1}<a_i$ which implies that $a_{i-1} \neq a_i$. Since   $a_{i-1} \leq a_i$, Condition $(\ast)$   in the definition of an additive poset implies that   $$b_i=a_{i-1} +a_{i} \leq a_i= a_{i-1}+b_i. $$
  Thus,  $b_i \leq a_i$ and    $b_i$  is independent from $a_{i-1}$  for all $i\geq 1$. Now, for  any $j=1, \ldots, i-1$, we have $a_j\leq  a_{i-1}$.    Since $b_i$  is independent from $a_{i-1}$,  Lemma~\ref{lemma4}  implies that
  $b_i$ is independent from $a_j$. Since $b_j \leq a_j$, Lemma~\ref{lemma4} implies that $b_i$   is independent from $b_j$. Note also that $b_i$ is independent from $a_0=a$ so that $b_i \in A^a$. This proves our  claim above.

  Conversely, consider a sequence  $b_1,..., b_n$ of   pairwise independent nonzero vectors in $A^a$.  Set $a_0=a$ and  $a_i =a+ b_1 + \ldots   + b_i $
   for   $i=1,..., n$.  This yields    a chain \eqref{chain} in~$A$.
  Indeed, by the assumptions, the vectors $a, b_1, ..., b_{i-1}$ are independent from $b_i$ for all $i$, and therefore, by Lemma~\ref{lemma4},    their sum  $a_{i-1}$ is independent from $b_i$. Thus,
  $a_{i-1} \leq a_{i-1} + b_i=a_i $. Also,  $a_{i-1} \neq a_i$   because  $b_i \neq 0$.
It is clear that the  correspondences above are mutually inverse.
\end{proof}

As an exercise, the reader may prove that for any independent elements $a,b$ of an additive poset~$A$, we have
  $A_{a+b} \supset A_{a} + A_b $ and $A^{a+b}=A^a \cap A^b$.

\section{Atoms and tiles}

\subsection{Atoms}  One says that
  an element~$a$ of a poset~$P$ \emph{covers} an element $b \in  P$ if $b<a$ and there is no $c \in P$ such that $b<c<a$.
 An element  of an additive poset  is an \emph{atom}  if it covers $0$.   In other words, an element    $a $ of an additive poset    is  an atom if  $a \neq 0$ and its tail  consists solely of~$0$ and~$a$. For example, all nonzero vectors in an additive poset with trivial partial order
  are atoms. 

  \begin{lemma}\label{lemma7}  An  element $a$ of an additive poset $A$ covers an element $b \in A$ if and only if  $a+b$  is an atom   and $a+b \leq a$.  \end{lemma}

\begin{proof} Suppose   that   $a+b$  is an atom   and $a+b \leq a$. Then $a+b \neq 0$ and so $a \neq b$.  By Condition $(\ast)$, the relations $a+b \leq a$ and $a  \leq a$  imply that $b\leq a$. Thus $b<a$.  
If there is    $c \in A$ such that $b<c<a$,  then applying  Lemma~\ref{lemma--5} to the chain  $  b <c< a$ we obtain  that   $b+c$ and $c+a$ are independent nonzero vectors.
Hence $$0< b+c < (b+c) +(c+a) =a+b.$$  This contradicts the assumption that $a+b $ is an atom.  Thus,  $a$ covers $b$.

Conversely, suppose   that $a$ covers $b$. Then $b<a$ and, in particular,  $ a\neq b$. Therefore  $a+b \neq 0$.  By Condition $(\ast)$, the relations     $a\leq a$ and $b\leq a$
imply that $a+b \leq a$. It remains to prove that $a+b$ is  an atom. If   not, then there is    $c \in A$ such that $0<c<a+b$.  Combining with   $ a+b \leq a$, we obtain that $c<a$. Using  Condition $(\ast)$ and the relations  $b<a$ and  $c<a$, we obtain  that $b+c \leq a$. Since $ c\neq a+b$, we have $b+c \neq a$. Thus $b+c<a$.  Next, using Condition $(\ast \ast)$, we deduce from  $c<a+b$ and  $c<a$  that $c \leq b+c$. Thus, $b$ and $c$ are independent and consequently,   $b \leq b+c$. Since $c \neq 0$, we have $b \neq b+c$. Thus, $b<b+c <a$. This contradicts the assumption that~$a$ covers~$b$.
Hence,   $a+ b$ is an atom.
\end{proof}

 Lemma \ref{lemma7} may be rephrased in terms of the Hasse diagrams. The   Hasse diagram of a poset~$P$  is a directed graph with the set of vertices~$P$  which has an edge directed from $a  \in P$ to $b\in P$      if and only if~$a $ covers~$b$. By the Hasse diagram of an additive poset  we mean the Hasse diagram of the underlying poset.   Lemma~\ref{lemma7} implies that   the elements of an additive poset $A$ covered by a vector $a\in A$ bijectively correspond to the atoms  of the additive  poset $A_a\subset A$.
The correspondence carries any $b\in A$ covered by $a$ to   $a+b   \in A_a$.    The inverse correspondence carries   an atom  $c\in A_a$ to $a+c $. Similarly, the elements of~$A$ covering a vector  $b\in A$ bijectively correspond to  the atoms  of   $A^b\subset A$.  The correspondence carries any $a\in A$ covering~$b$ to   $a+b   \in A^b$.    The inverse correspondence carries   an atom  $c\in A^b$ to $b+c $.

\subsection{Tiles}     An element   $a $ of an additive poset $  A$  is  a  \emph{tile} if $a \neq 0$ and any two distinct atoms  belonging to  $A_a \subset A$  are independent.  Clearly, all atoms   are tiles.  All nonzero elements of the tail of a tile are tiles.


\begin{theor}\label{tiltsy}  Let $a$ be a tile in an additive poset $A$.   Let $I$   be the set of all     atoms of   $A_a$.   The additive  subposet   of $A_a$ additively generated by~$I$    is isomorphic to   $2^I_f$. \end{theor}

\begin{proof}    For each  finite  set $J \subset I$, put $\varphi(J)=\sum_{j\in J} j  \in A_a$. 
This  defines an additive  homomorphism $\varphi: 2^I_f \to A_a$. It    is injective as directly follows from the assumption that $a$ is a tile and   Lemma~\ref{lemma--4}.     We need only to verify that~$\varphi$ carries the partial order in $2^I_f$ determined by the inclusion of sets  into the partial order in   $\varphi(2^I_f)$ induced by that of $A$.
Observe  that if   finite sets $J, L \subset I $ are disjoint, then all elements of $J$ are independent from   elements of $L$. By  the first claim of Lemma~\ref{lemma4}, the vectors $\varphi(J)$ and $ \varphi(L)  $   are independent from each other. Thus, $\varphi(J) \leq \varphi(J) + \varphi(L)= \varphi(J \cup  L)$.  Consequently,   for any  finite sets
$J\subset K \subset I $, we have $\varphi(J)  \leq \varphi(K)$.  Conversely, suppose that $\varphi(J)  \leq \varphi(K)$ for some finite  sets
$J, K \subset I $. For any $j_0\in J$, we have   $$   j _0\leq  \sum_{j\in J} j =\varphi(J)  \leq \varphi(K). $$ Therefore $ j_0 \leq \varphi(K)$.
If $j_0 \notin K$, then the vectors $ j_0$ and $\varphi(K)=\sum_{k \in K} k$ are independent and therefore incomparable.
This contradicts the relation  $ j_0\leq \varphi(K)$. Thus,  $j_0\in K$. This proves that $J\subset K$.
 \end{proof}

 \subsection{Examples}\label{exa23-}     For a  set $I$,   the atoms of the additive posets $2^I$ and $2^I_f$   are the one-element subsets of $I$.  All nonzero elements of $2^I$ and  $2^I_f $ are tiles.
 The  atoms of   $2^I_{ev}$   are the two-element subsets of~$I$.
 The additive poset $2^I_{ev}$    has no tiles other than its atoms.

   \subsection{Remark}\label{isi}   For  a finite set~$I$,   the  additive poset   $2^I_{ev}$ is isomorphic to an additive powerset if and only if $\vert I\vert \leq 2$.  Indeed, if $ I =\emptyset$ or $\vert I \vert =1$, then        $2^I_{ev}=\{0\}= 2^\emptyset$. If $\vert I \vert =2$, then     $2^I_{ev}=2^J$ where $J$ is a one-element set. If $\vert I \vert  =3$ or, more generally, if $\vert I \vert \geq 3$ is odd, then $2^I_{ev}$ does not have a greatest element while any additive powerset
  has a greatest element. If   $\vert I \vert  \geq 4$, then $2^I_{ev}$ has  nonzero  elements that are not tiles while all nonzero   elements of a finite additive powerset  are tiles.

\section{Finite additive posets}\label{Height and the independence relation}

An additive poset is    \emph{finite} if its underlying set is finite. In this section, we discuss  properties of   finite additive posets.

\subsection{Atoms as generators}  \label{Atoms as generators}        The following theorem shows       that every element of a finite additive poset   expands in a canonical way as a sum of  the atoms.

\begin{theor}\label{thm3}   Let $A $ be a finite additive poset and let $\mu: A \times A \to \Z$ be the M\"obius function of the partial order  in~$A$.    Then    any    $a\in  A$ expands    $a=\sum_{b\in {\mathcal A} } \mu(b,a) \, b$ where ${\mathcal A} \subset A$ is  the set of   atoms of~$A$. \end{theor}

\begin{proof} The  function $\mu$ is uniquely characterized by the following properties:           $\mu(a,b)=0$ for all   $a,b\in A$ such that $a \not \leq b$ and $\mu(a,a)=1$ for all $a \in A$;  $\mu(a,b)=-\sum_{a\leq c < b} \mu(a,c) $ for all   $a,b\in A$ with $a< b$.
Given a   map $f$ from $A$ to an abelian group~$B$, one defines a  map ${f_\bullet}:A \to B$ by
$$ {f_\bullet}(a)=\sum_{b\leq a} f(b)\quad \text{for all} \quad a\in A. $$
 The   M\"obius inversion formula says that
\begin{equation} \label{mobiu}f(a)= \sum_{b\in A}  \, \mu(b,a ) \,  {f_\bullet}(b) \quad \text{for all} \quad a\in A.\end{equation}
We apply these formulas  to the identity map $f=\id:A\to A$. The map ${\id_\bullet}$ carries any $a\in A$ to   $\sum_{b\in A_a} b$. If $a$ is  an atom, then ${\id_\bullet}(a) =a$. If~$a$ is  not an atom, then ${\id_\bullet}(a)=0$ because the vector  $\sum_{b\in A_a} b  \in A_a$ is  invariant under all   automorphisms of    the $\Z/2\Z$-vector space $A_a$ and $\dim_{\Z/2\Z} A_a \geq 2$. Now,  Formula  \eqref{mobiu} directly implies the  claim of the theorem.
\end{proof}

  \begin{corol}\label{corolla2}   Any finite additive poset  is additively generated      by its atoms.       \end{corol}

The following theorem yields a different kind of expansions of   elements of a finite additive poset     as   sums of  atoms.

\begin{theor}\label{thm4}  Every  nonzero element  $a$ of a finite additive poset $ A$  expands as a sum of pairwise independent atoms of~$A$.  The atoms in  any such   expansion of~$a$ belong to the tail of~$a$. \end{theor}

\begin{proof}     One says that a chain  $ a_0< a_1< \cdots <a_n$  in a poset is   \emph{saturated}  if $a_i$ covers $a_{i-1}$ for  $i=1,..., n$. Lemma~\ref{lemma7} implies that under the bijective correspondence of Lemma~\ref{lemma--5},    saturated chains in $A$ starting with $a_0=0$  correspond to sequences of
   pairwise independent atoms $b_1, \ldots, b_n$ in $A^0=A$.  The maximal element $a_n$ of the chain is computed by $a_n= b_1 + \cdots +b_n$.  Since $A$ is finite, for any $a\in A$ there is a saturated chain with maximal element  $a$. For the corresponding  pairwise independent  atoms $b_1, \ldots, b_n$, we have  $a= b_1 + \cdots +b_n$.

   To prove the second claim of the theorem consider an expansion of $a$ as a sum   of pairwise independent atoms $a_1,..., a_n \in A$ with $n\geq 1$.   If $n=1$, then $a_1=a\leq a$. If $n\geq 2$, then Lemma \ref{lemma4} implies that   $a_1$ is independent from   $b=a_2+ \cdots  + a_n$ so that  $a_1\leq a_1 +b=a$.  Similarly,  $a_i\leq a$ for all $i=1,..., n$.
 \end{proof}

 \subsection{Remarks}\label{remsheight} 1.  For any atom $a\in A$,   the expansions of~$a$ in both   Theorems~\ref{thm3} and~\ref{thm4} are just   $a=a$.

 2. The  expansion in Theorem \ref{thm4} is  not necessarily unique.  For instance, given four   distinct  elements $i,j,k,l$  of a set $I$, the element $\{i,j,k,l\}$ of the additive poset $2^I_{ev}$  expands  as a sum of pairwise independent atoms
in three   ways:
$$ \{i,j,k,l\}=\{i,j\} + \{k,l\}=\{i,k\} + \{j,l\}=\{i,l\} + \{j,k\} .$$


\subsection{Tiles re-examined}  We discuss the tiles of   finite additive posets.

 \begin{theor}\label{tilestiles} Let $A$ be a finite additive poset. The following conditions on a nonzero vector $a \in A$ are equivalent:

  (i) $a$ is a tile;

  (ii) the tail  $A_a$ of $a$ is isomorphic to  an additive powerset;

  (iii)   all atoms of    $A_a$ are pairwise independent and their sum is equal to~$a$;

  (iv) the expansion of $a$ as   a sum   of pairwise independent atoms of $A$  is unique up to permutation  of the  summands.
  \end{theor}

\begin{proof}  The implication $(i) \Rightarrow (ii)$   follows from Theorem \ref{tiltsy} and the fact that the atoms of $A_a$ additively generate $A_a$. Note also that the set~$I$ of atoms of $A_a$ is finite and so $2^I_f=2^I$.


We prove that $(ii) \Rightarrow (iii)$.  Under an identification of $A_a$ with the additive powerset   of  a finite set~$I$, the atoms of $A_a$ correspond to singletons while the vector  $a\in A_a$, being the maximal element of $A_a$,  corresponds to the set $I$ itself. It is clear that the  singletons are pairwise disjoint and their union is~$I$.  This proves (iii).


 We   prove that $(iii) \Rightarrow (iv)$.  Given an  expansion of $a$ as   a sum   of pairwise independent atoms  $a_1,..., a_k$ of $A$, we have  $a_1,..., a_k \in A_a$ by Theorem~\ref{thm4}.  We need only to show that all atoms of $A_a$ appear  among $a_1,..., a_k$.    Let $b_1,..., b_l$ be the atoms of $A_a$ not appearing among  $a_1,..., a_k$. By   assumption, the atoms  $a_1,..., a_k, b_1,..., b_l$ are pairwise independent and their sum is equal to~$a$. Since $a =a_1+ \cdots +  a_k$, we have  $b_1+\cdots + b_l=0$. By Lemma \ref{lemma--4}, we have $l=0$.

 We now prove that $(iv) \Rightarrow (i)$. Suppose   that $a$ has a  unique expansion  as a sum   of pairwise independent atoms $a_1,..., a_n$ of $ A$ with $n\geq 1$.   By Theorem~\ref{thm4},    $a_1,..., a_n \in A_a$. Pick   any atom $c\in A_a$. If $a=c$, then~$a$ is an atom  and $(i)$ is clear. Suppose that     $ a \neq c$ and set $b=a+c \neq 0$.
Since $c\leq a=c+b$, the atom~$c$ is independent from~$b $.  Pick an  expansion of~$b$ as a sum of pairwise independent atoms $ b_1,...,   b_m \in A_b$ with  $m\geq 1$.   Lemma \ref{lemma4} implies that   $c$ is independent from  $b_j $ for all~$j $. Thus, the atoms $c, b_1,..., b_m$ are pairwise independent and their sum is equal to $c+b=a$. By   assumption, the sets of atoms $\{a_1,..., a_n\}$
and $\{c, b_1,..., b_m\}$ must coincide. Thus, $c=a_i$ for some $i$.  Therefore,  all atoms of $A_a$ belong to the set $\{a_1,..., a_n\}$. This   implies  that $a$ is a tile. \end{proof}

  \subsection{Example}\label{exa23}
   Let $ a_0\in H\subset A$ be as in  Example~\ref{exam22}.3 and   $\dim(H)\geq 2$. Then $A_{a_0}=H$ and the  set
of  atoms of $A_{a_0}$ is  $\mathcal A= H \setminus \{a_0, 0\}$. The expansion of $a_0$  in   Theorem~\ref{thm3}
is $a_0= \sum_{a\in \mathcal A  } a$.     Theorem~\ref{thm4} yields an expansion of~$a_0$   as a sum of    two atoms: any   $b\in \mathcal A$ and $a_0+b \in \mathcal A$.  
The vector $a_0 $ is a tile iff $\dim(H)= 2$.

\section{Invariants of finite additive posets}


\subsection{Invariants}\label{finite additive posets}  We  consider four numerical     invariants of finite additive posets: the height,   the width,  the weight, and the dimension.
The first two   are defined for any finite poset~$P$.
The \emph{height} $h(P) $   is the maximal length of a chain in~$P$.  The \emph{width} $w(P)$   is the cardinality of a biggest antichain in~$P$ where   an antichain    is a subset  consisting of
pairwise incomparable elements.
The height $h(A)$ and   width  $w(A)$ of a finite additive poset~$A$ are   the height and width of the    underlying poset.  The \emph{weight} $wt(A)$ of   $A$ is the number of atoms of~$A$.   The \emph{dimension} $\dim(A)$ of~$A$ is the dimension of~$A$ as a $ \Z/2\Z$-vector space. Clearly, $\dim(A) =\log_2 \vert A \vert  $ where the vertical bars stand for the number of elements  of a   set.


\begin{theor}\label{heightind}   For any finite additive poset $A$, we have $$h(A) \leq \dim(A) \leq wt(A) \leq w(A).$$  \end{theor}

\begin{proof}    By definition,   $n=h(A)$ is the maximal integer   such that    $A$ has a chain of length~$n$.
 Since $0\in A$ is the  least element of $A$,  a   chain of   length~$n$ in~$A$   must  start   with~$0$. Lemma \ref{lemma--5}  implies  that
   $n$ is the maximal integer such that  there are~$n$   pairwise independent nonzero vectors in $A^0=A$.
 By Lemma~\ref{lemma--4},   $n\leq \dim(A)$. The inequalities $\dim(A) \leq wt(A) \leq w(A)$ follow  from the fact that the atoms of~$A$ generate~$A$ as a vector space and  form an antichain.
 \end{proof}

 \subsection{Examples}\label{SSpS}  1. Let $I$ be a finite set with $n $ elements. Then $$h(2^I)=\dim(2^I)=wt(2^I)=n$$
and, by  Sperner's theorem (see \cite{En}),   $w(2^{I})={n \choose [n/2]}$.

2. By Theorem \ref{heightind}, if~$P$  is the underlying   poset of a finite additive poset, then
 $h(P) \leq \log_2 (\vert P \vert )  \leq w(P)$.  For example, for an  integer $n \geq 1$, the  set $\{0,1, \ldots, n\}$ with partial order  $0\leq 1 \leq \cdots \leq n$ has  the height $n$. Since   $n  \leq \log_2 (n+1 )  $  only for $n=1$,
   the poset $\{0,1, \ldots, n\}$  with   $n\geq 2$ does not underlie  an additive poset.


     \subsection{The weight function}  Let $A $  be a finite additive poset.  Any invariant $\psi$ of finite additive posets determines two functions on $A$   by
     $\psi(a)=\psi(A_a) $ and $ co-\psi (a)=\psi (A^a)$ for all $a\in A$.        In particular,    these definitions apply to the invariants   from  Section \ref{finite additive posets}.     We briefly discuss the functions associated with $\psi=wt$.

We define
    the \emph{weight} of    $a \in A$   by $wt(a)=wt(A_a)$. Thus, $wt (a)$ is the number of  atoms of~$A$ that are smaller than or equal to~$a$.   In other words,   $wt(a)$ is   the number of edges in the Hasse diagram  of $A$ directed   from $a$ to other vertices.  If $a,b\in A$ satisfy $a <b$, then  $wt(a) <wt(b)$. Indeed, since $A_a\subset A_b$, all atoms of $A_a$ are also atoms of $A_b$. By Corollary~\ref{corolla2}, $A_b$ is generated by its atoms.   Since $b\in A_b \setminus A_a$, we have $A_b\neq A_a$, and thus at least one    atom  of $A_b$ does not belong to $A_a$.


We  define the \emph{coweight} of $a\in A$  by $cowt(a)=wt(A^a)$. Thus, $cowt(a)$ is the number of atoms of~$A$ independent from~$a$ or, equivalently,    the number of edges of the Hasse diagram  of $A$  directed to~$a$. Since $A_a \cap A^a=\{0\}$, we have
 $wt(a)+ cowt(a) \leq wt(A)$.
 If $a,b\in A$ satisfy $a <b$, then  $cowt(a) > cowt(b)$. This is shown   as in the previous paragraph using that $A^a\supset  A^b$ and $a+b \in A^a\setminus A^b$.


\section{Plain additive posets}


We introduce   a class of     plain   additive posets.

\subsection{Embeddings}\label{Small additive posets}
 An  \emph{embedding} of an additive poset $A$ into an additive poset~$B$ is an isomorphism of~$A$ onto an additive subposet of~$B$. In other words, an embedding $\varphi : A \to B$ is a    group monomorphism  such that for any $a,b \in A$, the relation $a\leq b$  holds in $A$ if and only if  $\varphi (a) \leq \varphi(b)$   in $B$.
 We say that~$A$   \emph{embeds} in~$B$ if there is an embedding   $  A \to B$.

   We say that an additive poset $A$   is \emph{plain} if it embeds in the additive powerset $2^{I }$ for some  finite set~$I$.   Then $A$ is finite and
 $ \dim (A) \leq \dim(2^{I }) =\vert I\vert $.

 \subsection{Separating   functionals}\label{Separating sets of functionals}
It is useful to reformulate plainness     in terms of linear functionals.
By a linear functional on an additive poset~$A$, we mean a group homomorphism $A \to \Z/2\Z$.
A  linear functional  $f : A \to \Z/2\Z$ is \emph{order-preserving} if it is a morphism of additive posets $A \to (\Z/2\Z, \leq_t)$. In other words,  $f $ is  order-preserving  if for any $a,b\in A$ such that $a \leq b$, we have  $f(a) \leq_t f(b)$.

We say that a set $S \subset   A^*=\Hom (A, \Z/2\Z)$  is   \emph{separating} if
 all elements of $S$ are order-preserving   and   for any $a,b\in A$  with $a\not \leq b$, there is   $s\in S$ such that $s(a)=1$ and $s(b)=0$.  Taking here $b=0$ we obtain that  for any nonzero $a\in A$,    there is  $s\in S$ such that $s(a)=1$. In other words, $\cap_{s\in S} \Ker s=0$.

  \begin{theor}\label{simple} The following conditions on a finite additive poset $A$ are equivalent:

  (i) $A$ is plain;

  (ii) The set of  all order-preserving linear functionals on~$A$ is separating;

  (iii) There exists a separating subset of~$A^*$;

  (iv) The partial order in~$A$ is obtained   as in  Theorem~\ref{presd} from  a set $S\subset   A^*$ such that $\cap_{s\in S} \Ker s=0$.
  \end{theor}

\begin{proof}   We  prove that $(i) \Rightarrow (ii)$. Suppose   that there is  an embedding $\varphi : A \to 2^{I }$, where $I $ is a finite set. For each $i\in I $, consider the  map $s_i:2^I \to \Z/2\Z$
which carries  a set $J \subset I$ to $ 1$ if $i\in J$ and to $ 0$ if $i \not \in J$.    It is clear  that $s_i$ is an order-preserving linear map.     Hence,  $s_i \varphi: A \to \Z/2\Z$ is an order-preserving linear functional.  To check (ii),   pick    any $a,b\in A$  with $a\not \leq b$. Then
 $\varphi(a)\not \leq \varphi(b)$, i.e., the set $\varphi(a) \subset I$ is not contained in  the set $\varphi(b)\subset I$. For  any  $i\in \varphi(a) \setminus \varphi(b)$, we have  $s_i\varphi(a)=1$ and $s_i\varphi(b)=0$.

The implication  $(ii) \Rightarrow (iii)$  is obvious.
We show that  $(iii) \Rightarrow (iv)$. Let $S \subset A^*$ be a  separating set. As we know,  $\cap_{s\in S} \Ker s=0$.
  We claim that the given partial order $\leq$ in~$A$ coincides with the partial order $\preceq $ in~$A$ determined by~$S$ as in Theorem~\ref{presd}. Indeed, pick any $a,b\in A$. If $a\leq b$, then $s(a) \leq_t s(b)$ for all $s\in S$ because all  $s\in S$ are order-preserving. By the definition of $\preceq $, we have $a\preceq b$. If $a \not \leq b$, then, since $S$ is separating,   there is $s\in S$ such that $s(a)=1$ and $s(b)=0$. By  the definition of $\preceq $, we have $a \not \preceq b$.
Therefore, $\leq \,  =\, \preceq $.


It remains to prove that $(iv) \Rightarrow (i)$.    Suppose   that the partial order   in~$A$ is obtained    as in  Theorem~\ref{presd} from a set $S\subset A^*$  such that $\cap_{s\in S} \Ker s=0$.  Consider the map $ \varphi: A \to 2^S$ carrying any $a\in A$ to     $$\varphi(a)=\{s\in S\, \vert \, s(a)=1\}\subset S.$$
 The   map $\varphi $  is an additive  homomorphism: for any $a,b\in A$,
$$\varphi(a+b)= \{s\in S\, \vert \, s(a+b)=1\}=  \{s\in S\, \vert \, s(a)+ s(b)=1\}$$
$$=\{s\in S\, \vert \, s(a)=1, s(b)= 0 \quad \text{or} \quad  s(a)=0, s(b)=1\}$$
$$=\{s\in S \, \vert \, s(a)=1\} + \{s\in S\, \vert \, s(b)=1\} = \varphi(a) + \varphi(b).$$
  The condition $\cap_{s\in S} \Ker s=0$ ensures that if $a \neq 0$, then $\varphi(a)\neq \emptyset$. Thus,  $\varphi$ is a monomorphism. Next, consider any $a,b\in A$. If $a\leq b$, then  $s(a) \leq_t s(b)$ for all $s\in S$. Thus, $s(a)=1 \Rightarrow s(b)=1$.
Consequently, $\varphi(a) \leq \varphi(b)$. If $a \not \leq b$, then there is  $s\in S$ such that $s(a)=1$ and $s(b)=0$.
Then  $s\in \varphi(a)$  and $s \not\in \varphi(b)$. Consequently, $\varphi(a) \not \leq \varphi(b)$. We conclude that  the map $\varphi:A\to 2^S$ is an embedding of additive posets.
 \end{proof}

%
%
%


\subsection{Complexity}
The \emph{complexity}  $c(A)$  of a plain additive poset  $A$ is the smallest integer $n \geq 0$ such that $A$ embeds in  the additive powerset  $2^{I }$ for an $n$-element set~$I$.
Note that all additive subposets of~$A$ are plain   and have complexity $  \leq c(A)$.
 If   $A$ is isomorphic to $2^I$ for a finite set~$I$, then $c(A) = \dim(A) =\vert I \vert  $. If   $A$ is not isomorphic to an additive powerset, then   $c(A) \geq \dim(A)+1  $.



 \begin{theor}\label{thm17}  For any plain additive poset $A$, we have $c(A)=\min_S \vert S\vert$
where $S$ runs over all separating subsets of $A^*$ and \begin{equation}\label{sperner}  {c(A) \choose [c(A)/2]} \geq w(A),\end{equation}
where $w(A)$ is the width of $A$.  \end{theor}

\begin{proof} An embedding  $A \to 2^{I}$
yields a separating subset  of $A^*$  consisting of  $\vert I \vert$ elements, see the   proof of the implication $(i) \Rightarrow (ii)$ in   Theorem~\ref{simple}. Hence  $\min_S \vert S\vert \leq c(A)$. Also, for any separating  set $S \subset A^*$, there is an embedding  $A\to 2^{S}$, see the   proof of the implications $(iii) \Rightarrow (iv) \Rightarrow (i)$ in  Theorem~\ref{simple}. Thus, $c(A)\leq \min_S \vert S\vert$.
To prove \eqref{sperner}, it is enough to observe that if $A$ embeds in $2^I$, then $w(A) \leq w(2^I)={\vert I \vert \choose [\vert I \vert/2]}$, cf. Example~\ref{SSpS}.1. \end{proof}





\subsection{Example}     We compute      $ c(2^I _{ev})$ for any finite set~$I$. If $I=\emptyset$ or $\vert I\vert =1$, then    $2^I_{ev}= 2^\emptyset$ and  $c(2^I_{ev})  =0$.
 If $\vert I \vert =2$, then
$2^I_{ev}= 2^K$, where $K$ is a one-element set, and   $c(2^I_{ev})  =1$. Assume  that $\vert  I \vert \geq 3$, pick   $i\in I$, and set $K=I \setminus  \{ i \}$. The  map
$2^I _{ev}\to 2^K, J \mapsto J \cap K$ is   a  group  isomorphism  and therefore 
  $$\dim(2^I _{ev})=\dim(2^K)=\vert K \vert=\vert I \vert -1.$$ Since $2^I_{ev}$ is an additive subposet of  $2^I$ and    $2^I_{ev}$ is not   isomorphic to an additive powerset  (Remark \ref{isi}), we have    $$  \vert I \vert \geq  c(2^I_{ev}) \geq \dim(2^I_{ev}) +1 = \vert I \vert. $$
  We   conclude that
  $c(2^I_{ev})=\vert I \vert $.




\section{Homology posets}\label{Additive posets in topology}



\subsection{Basics}\label{CW-complexes} We define the structure of an additive poset in    the  top-dimensional homology of any finite-dimensional CW-complex. We first recall   the   terminology, see \cite{LW}. A \emph{$k$-ball} $D^k $ with  $k= 1, 2, \ldots$ is a      ball  in Euclidean space $  \RR^k$.   To attach    $D^k $ to a topological space~$Y$ along a (continuous) map $\varphi:\partial D^{k}\to Y$, one takes the disjoint union $Y \amalg D^k$ and   identifies each point  of the sphere $\partial D^{k}=S^{k-1}$ with its image    under~$\varphi$. 
One   similarly attaches   families of  disjoint  $k$-balls to $Y$ along   maps of their boundary spheres to~ $ Y$.
For an  integer $n \geq 0$, an \emph{$n$-dimensional CW-complex} $X$  is a  Hausdorff   topological space endowed with a filtration $$X^{(0)} \subset X^{(1)} \subset \cdots   \subset X^{(n)} =X,$$
where $X^{(0)} $ is a discrete set of points and for $k=1, \ldots, n$, the space $X^{(k)} $ is obtained by attaching a family of $k$-balls to $X^{(k-1)} $.  One calls $X^{(k)} $   the \emph{$k$-skeleton} of~$X$.
The space  $X^{(k )} \setminus X^{(k-1)} $ is  a disjoint union of copies of $\Int(D^k)=D^k\setminus S^{k-1}$ called the \emph{(open) $k$-cells} of~$X$.
Let $I_k$ be the set of all   $k$-cells of~$X$. For  any cells $e\in I_k$ and $e'\in I_{k-1}$, one defines  a residue $[e :e'] \in \Z/2\Z$ as follows. 
Compose the attaching map  $ S^{k-1} \to X^{(k-1)}$ of~$e$ with the map $  X^{(k-1)} \to S^{k-1}$ obtained by collapsing $X^{(k-1)}\setminus e'$ to a point and identifying the result  with $S^{k-1}$. Then $[e:e']$ is the degree mod  2 of the composed  map $  S^{k-1} \to  S^{k-1} $.

The top   homology $ H=H_n(X; \Z/2\Z)$ of~$X$ is   the  subspace  of the  $\Z/2\Z$-vector space   $ 2^{I_n}$ consisting of   all sets $E\subset I_n$ such that $\sum_{e\in E} [e:e'] =0$
  for all $e'\in I_{n-1}$. Such sets~$E$ are called   \emph{(cellular) $n$-cycles}.
The   partial order in $ 2^{I_n}$ determined by the inclusion of  sets     restricts to a partial order   in~$H$  and turns~$H$
 into an additive poset. We call $H$  the \emph{homology poset} of $X$.


The homology poset of a CW-complex $X$  is  invariant under subdivisions of~$X$. Recall that a CW-complex~$Y$ is a \emph{subdivision} of~$X$  if~$X$ and~$Y$ share  the same underlying topological space and every   open cell  of~$X$ is a union of a finite number of open cells of~$Y$ (of the same or smaller dimension). For any $n$-cycle $E$ of~$X$,    the   $n$-cells of~$Y$ lying in the $n$-cells   belonging to~$E$ form an $n$-cycle of~$Y$. This defines an isomorphism of additive posets $H_n(X;\Z/2\Z)\simeq H_n(Y;\Z/2\Z)$.


 \subsection{Representation of homology classes} We relate the partial order in homology  to the  problem of representation of  homology classes by manifolds. For $n \geq 0$, by a closed $n$-manifold we mean a non-empty compact $n$-dimensional topological manifold with void  boundary. Given a closed $n$-manifold $M$, we let $[M]\in H_n(M; \Z/2\Z)$ be its fundamental class.   A homology class $a\in H_n(X;\Z/2\Z)$ of a CW-complex $X$  is \emph{represented} by a (continuous) map $i:M\to X$ if $i_\ast([M])=a$.

    \begin{theor}\label{repr-}  Let~$X$ be an $n$-dimensional CW-complex with $n\geq 0$. Let $a, b\in H=H_n(X;\Z/2\Z)$ be  homology classes    represented respectively by maps $i:M\to X$ and $j:N\to X$, where $M,N$ are  closed   $n$-manifolds. If $i(M) \cap j(N)=\emptyset$, then    $a$ and $b$ are independent elements of the additive poset~$H$.   \end{theor}

    \begin{proof} Since~$M$ and~$N$ are compact, so are their images   $i(M) \subset X$ and $j(N) \subset X$. Since $X$ is Hausdorff, both   $i(M)$ and $j(N)$ are closed. If they are disjoint, then they  have disjoint open neighborhoods $i(M) \subset U$ and $j(N) \subset V$ (we use  the fact that all CW-complexes are normal). Taking a sufficiently small subdivision of $X$, we can assume that all  $n$-cells of $X$ meeting $i(M)$ are  contained in $U$ and all   $n$-cells of $X$ meeting $j(N)$ are  contained in $V$. Then the homology classes $a=i_\ast([M])$ and $b=j_\ast([N])$ are represented by disjoint $n$-cycles. The union of these   $n$-cycles represents $a+b$. By the definition of the partial order in $H=H_n(X;\Z/2\Z)$, we have $a\leq a+b$, so that $a$ and $b$ are independent.
       \end{proof}


 An \emph{embedding} of a closed manifold~$M$ into  a CW-complex $X$ is  an injective (continuous) map $M \to X$. Since $M$ is compact and $X$ is Hausdorff, such a map is a homeomorphism onto its image.

 \begin{theor}\label{repr}  Let~$X$ be an $n$-dimensional CW-complex with  $n\geq 0$.
 If  a homology class  $a\in H=H_n(X;\Z/2\Z)$ is represented by an embedding of a closed  $n$-manifold $M$ into~$X$, then~$a$ is a tile.
 Moreover, if $M$ is connected, then $a$ is an atom.     \end{theor}

    \begin{proof} Let  $i:M \to X$ be an embedding representing~$a$. As in the proof of Theorem~\ref{repr-}, the set  $i(M)$   is  closed in $X$.    Therefore, for any (open) $n$-cell $e$ of~$X$,  the set $e \cap i(M)$ is closed in~$e$. It is also open in~$e$, as directly follows from the assumption that~$i$ is an embedding and~$M$ is a closed $n$-manifold.     Since~$e$ is connected,   either $e \cap i(M)= \emptyset$ or $e \subset i(M)$. It is clear then that the $n$-cells   of~$X$  contained  in
    $  i(M)$ form an $n$-cycle, say $E$,  representing   $a=i_\ast([M]) \in H$. Since $M \neq \emptyset$ and $i$ is an embedding, $E\neq \emptyset$ and so $a\neq 0$.

    We claim  that if $M$ is connected, then
  $a$ is an atom. We need  to show that any   $b \in H$ satisfying  $b<a$ is  equal to~$0$. Such~$b$ is  represented by an  $n$-cycle formed by some (but not all)   $n$-cells   of~$X$  contained in $i(M)$. These $n$-cells are    contained in $i(M \setminus
  \{x\})$ for some  $x\in M$. Hence,  $b$ lies in the image of the homomorphism $$H_n(M \setminus  \{x\};\Z/2\Z)  \to H=H_n(X;\Z/2\Z)  $$
   induced by  the restriction  of~$i$ to $M \setminus  \{x\}$. Since $M$ is a connected $n$-dimensional manifold, $H_n(M \setminus  \{x\};\Z/2\Z)=0$. Therefore $b=0$.

   Suppose now that $M$ has $m\geq 2$ connected components $M_1,..., M_m$. For $k=1,...,m$, set $a_k=i_\ast ([M_k])\in H$.
 Then  $$a= i_\ast ([M])= i_\ast ([M_1]+ \cdots + [M_m]) =a_1+ \cdots + a_m.$$
   By Theorem~\ref{repr-} and the previous paragraph, this is an  expansion of~$a$ as a sum of pairwise independent atoms.
   To prove that $a$ is a tile, it is enough to show that every atom $b\leq a$ coincides with one of the atoms $a_1,..., a_m$.
   By the above, $a=i_\ast([M])$ is represented by the  $n$-cycle $E$ consisting of all  $n$-cells   of~$X$ contained in
    $  i(M)$. The atom~$b$ is then represented by an $n$-cycle $F\subset E$.   The set~$F$  splits as  a disjoint union  $ \amalg_{k=1}^m       F_k$ where  $F_k$ is the set of     $n$-cells  of~$X$ which belong to~$F$ and are contained   in
    $ i(M_k)$. Since the sets $ \{i(M_k)\}_{k=1}^m$ are closed and pairwise disjoint, the assumption that $F$ is an $n$-cycle implies that  $F_k$ is an $n$-cycle for all $k$.  The arguments above show that either $F_k=\emptyset$   represents $0\in H  $ or   $F_k$ includes all   $n$-cells  of~$X$   contained   in
    $ i(M_k)$ and represents   $a_k \in H$.     Consequently,  the homology class~$b$ represented by~$F$ is a sum of several   classes $a_k$. Since these classes are   independent, $a_k \leq b$ for all~$k$ in question.   Since~$b$ is an atom, $b=a_k$ for some $k$.
       \end{proof}

       \subsection{Examples}\label{exewed}  1. Let $X$ be a   wedge of several $n$-dimensional spheres  with  $n \geq 1$. The homology classes of these spheres   form a basis,   $I$, of   $H_n(X;\Z/2\Z) $.  It is clear  from the definitions that the homology poset of $X$ is   nothing but  $2^I_f$.

       2.  Consider  a CW-complex $Y$ obtained by gluing several $n$-balls  along   their boundary spheres.
       The homology poset of $Y$ is     $2^I_{ev}$, where~$I$ is the set of   $n$-balls forming~$Y$.

        \subsection{Remarks}   1. In general, the   partial order in the homology poset   is not preserved  under homotopy equivalences of CW-complexes. Consider, for instance, the CW-complex $Y$ obtained by gluing three $n$-balls along their boundary spheres. Then $H_n(Y;\Z/2\Z) = (\Z/2\Z)^2$ with  trivial partial order. Collapsing one of the  balls into a point, we turn $Y$ into  a wedge~$X$ of two $n$-spheres, and the partial order   in $H_n(X;\Z/2\Z)$ is   non-trivial as is clear from Example \ref{exewed}.1.

2. The definition of the homology poset  of a CW-complex readily extends to CW-pairs and relative CW-complexes. We do not pursue this line.



\section{Geometric realization of additive posets}

\subsection{Realization theorem}\label{Realization  theorem} A    CW-complex is    \emph{finite} if it has   a finite number of cells. Clearly, the homology poset of a finite CW-complex $X$ is plain and its complexity  is smaller than or equal to the number  of  top-dimensional cells of~$X$.  

 \begin{theor}\label{realizadpos}   For any plain additive poset~$A$ and any integer $n \geq 2$, there is a finite   $n$-dimensional CW-complex whose homology poset is isomorphic to~$A$.
     \end{theor}

     The rest of this section presents a proof of Theorem~\ref{realizadpos}.

     \subsection{Cohomological computations}\label{Cohomological computation of the order}      Consider  an $n$-dimensional CW-complex~$X$. In Section~\ref{CW-complexes} we used inclusions of $n$-cycles to define a partial order  in 
 $H=H_n(X; \Z/2\Z)$. Here we compute this order in cohomological terms. Using the standard evaluation of cohomology classes on homology classes, we  identify $H^n(X; \Z/2\Z)$ with $H^*=\Hom(H, \Z/2\Z)$.
Each  (open) $n$-cell $e $ of~$X$  gives rise to a homotopy class of maps $s_e:X \to S^n$   obtained by collapsing $X \setminus e$ to a point  and identifying the resulting quotient space of~$X$ with the $n$-sphere $S^n$ via a  homeomorphism.  Set   $$X_e=(s_e)^* (z)\in H^n(X; \Z/2\Z) =H^* ,$$ where   $z $ is the    non-zero element of   $ H^n(S^n; \Z/2\Z)
=\Z/2\Z$.  We let $ S_X \subset  H^*$  be the set of all vectors in  $H^*$ associated in this way with the  $n$-cells of~$X$.  Note  that different $n$-cells of~$X$ may give rise to the same vector  and  that $\cap_{s\in S_X} \Ker s=0$. Indeed,  any non-zero $h\in H$ is represented by a  non-empty cellular $n$-cycle $E$ in~$X$, and then $ X_e(h)=1$ for all $e\in E$.    By Theorem~\ref{presd}, the set  $ S_X$  determines a partial order in~$H$   turning~$H$ into  an additive poset.
It follows from the definitions that this   partial order  in~$H$ coincides with the partial order defined   in Section~\ref{CW-complexes}.

  We say that a pair (a  $\Z/2\Z$-vector space $B$, a  set $S \subset B$) is \emph{realized} by an
  $n$-dimensional CW-complex $X$ if there is  an isomorphism of $\Z/2\Z$-vector spaces    $  H^n(X; \Z/2\Z) \to B$   carrying     $S_X$ onto $S$. The pair $(B, S )$    is \emph{realizable} if it  is realized by   a finite
  $n$-dimensional CW-complex. For example, the pair  (a  finite-dimensional   $\Z/2\Z$-vector space~$B$, a basis of~$B$)    is realized by a wedge of $n$-spheres.


 \begin{lemma}\label{lemma59}   For any $n \geq 2$, every pair $(B,S \subset B \setminus \{0\})$ where  $B$ is a  finite-dimensional   $\Z/2\Z$-vector space  and $S$ generates~$B$ as a vector space, is realizable. \end{lemma}

 \begin{proof}  We define two moves   $M_1, M_2$   transforming~$S $.   The move $M_1$ adds to $S$ a new  element  of the form $a+b$ for some distinct $a,b\in S$.
 The move $M_2$
 takes two distinct  $a,b\in S$,  adds   $a+b$  to~$S$, and deletes~$b$ from~$S$.  Note that in both cases   $a+b\neq 0$. If     $ a+b \in S$ before the   move, then $M_1$ keeps~$S$   and $M_2$    deletes~$b$ from~$S$.

 The lemma is a   consequence of  the following two claims:

 (a)   For any set $S \subset B \setminus \{0\}$ generating $B$ as a vector space, there is a sequence of moves $M_1, M_2$ transforming  a basis of $B$ into $S$;

  (b) If a pair $(B,S \subset B \setminus \{0\})$ is realizable and a set $S' \subset B \setminus \{0\}$ is obtained from $S$ by    $M_1$ or $ M_2$, then  the pair $(B,S')$ is realizable.

    We prove  (a). Since the set~$S$ generates~$B$,   it contains a basis $S_0$ of $B$. Each    $s\in S$ expands uniquely as a   sum  (without repetitions) of   vectors of  $S_0$. We call the number of summands the \emph{size} of~$s$.   Using   $M_1, M_2$, we   can consecutively add the vectors of $S \setminus S_0$ to~$S_0$ and thus    transform~$S_0$  into~$S$. For example, if   $s \in S \setminus S_0$ expands as $s=s_1+s_2+s_3$ with $s_1,s_2,s_3 \in S_0$, then we first apply $M_1$   adding   $s_1
  +s_2$    and then  apply $M_2$ adding $s_1+s_2+s_3$ and deleting $s_1
  +s_2$. To avoid  interactions between the  moves, we   first add   to $S_0$   the vectors of $S \setminus S_0$ of the maximal size,   then the vectors of the maximal size  minus~1,  etc.

  We now prove   (b). Suppose that the pair $(B,S)$ is realized by a finite $n$-dimensional CW-complex $X$. To simplify   notation, we identify $B$ with $  H^n(X; \Z/2\Z)  $  and~$S$ with $S_X$. Pick any distinct $a,b\in S$. Then $a=X_e$ and $b=X_f$ where~$e$ and~$f$ are  distinct (open) $n$-cells of $X$. Consider their closures   $\overline e \supset e$   and $\overline f \supset f$  in~$X$. Subdividing if necessary the $(n-1)$-skeleton $X^{(n-1)} $ of~$X$ (such subdivisions keep $  H^n(X; \Z/2\Z)  $  and  $S_X$) we can assume that $X$ has at  least one 0-cell $x_e$ lying in $\overline e \setminus e$ and at  least one 0-cell $x_f$ lying in $\overline f \setminus f$.  Consider an $n$-ball $D $  viewed as a CW-complex with one  $0$-cell $x\in \partial D $, one $(n-1)$-cell $\partial D\setminus \{x\}$, and one $n$-cell $\Int(D)= D \setminus \partial D$. Pick an embedding $i: D  \hookrightarrow \overline e$ such  that $i(x)=x_e$, $i(D  \setminus \{x\}) \subset  e$, and $ e'=e\setminus i(D)$ is an $n$-cell. Similarly, pick an embedding $j: D  \hookrightarrow \overline f$ such that $j(x)=x_f$, $j(D  \setminus \{x\}) \subset  f$, and $f'=f \setminus j(D)$ is an $n$-cell. We form a quotient space~$Y$ of~$X$ by identifying $i(d)=j(d)$ for all $d\in D$. The  space $Y$ is a finite CW-complex whose   $(n-1)$-skeleton   is obtained from     $X^{(n-1)} \amalg  \partial D$ by the identification   $x_e= x=x_f  $. The $n$-cells of~$Y$ are the images under the projection $X \to Y$ of the $n$-cells of $X$ distinct from $e,f$ together with $e'$, $f'$, and the $n$-cell $ g =i(\Int D)=j(\Int D) $.
  Considered up to homotopy equivalence,  $Y$ either  coincides with~$X$ (if $x_e=x_f \in X$) or is obtained from~$X$ by adjoining an arc connecting   $x_e$ to $x_f$ (if $x_e\neq x_f  $). Since $n \geq 2$, in both cases the projection $X \to Y$ induces an additive isomorphism   $H_n(X; \Z/2\Z)\to H_n(Y; \Z/2\Z)$. Identifying these two groups along this isomorphism, we obtain that  $Y_c=X_c$ for all $n$-cells~$c$ of~$X$ distinct from $e,f$ and
  $$Y_{e'}=X_{e}=a, \quad Y_{f'}=X_{f}=b, \quad Y_{g}=a+b.$$ Thus, $S_Y$ is obtained from $S=S_X$ by the move $M_1$. This proves the part of Claim (b) concerning $M_1$.


The part of Claim (b) concerning $M_2$ is proved similarly using the same $e,f,D, i$  as above but a different  map $j:D\to \overline f$. Observe that~$X$ can be  obtained from  $X \setminus f$ by attaching the $n$-ball $D$ along a  map $  \partial D  \to X \setminus f$.  We let $j $ be the  composition of the inclusion $D \hookrightarrow D \amalg (X \setminus f)$ with the projection  $D \amalg (X \setminus f) \to X$. Then~$j$ restricts to    a homeomorphism $\Int D \approx f $.  We form a quotient space~$Y$ of~$X$ by identifying $i(d)=j(d)$ for all $d\in D$.  As above, $Y$ is a finite CW-complex whose $n$-cells   are the images under the projection $X \to Y$ of the $n$-cells of $X$ distinct from $e$ together with $ e'=e\setminus i(D) $.
 The projection $X \to Y$ induces an additive isomorphism   $H_n(X; \Z/2\Z)\to H_n(Y; \Z/2\Z)$,  and identifying these two groups along this isomorphism we obtain that  $Y_c=X_c$ for all $n$-cells~$c$ of~$X$  distinct from~$e$ and~$f$, $Y_{e'}=a$,
   $Y_{f}=a+b$. Thus,  $S_Y$ is obtained from $S=S_X$ by the move $M_2$. This  completes the proof of Claim~(b) and of the lemma. \end{proof}

\subsection{Proof of Theorem~\ref{realizadpos}} Given a plain additive poset~$A$, Theorem~\ref{simple} implies that the partial order in~$A$ is determined  as in  Theorem~\ref{presd} by  a set $S\subset   A^*$ such that $\cap_{s\in S} \Ker s=0$. Eliminating if necessary the zero vector from~$S$, we can assume that $S\subset   A^* \setminus \{0\}$. Since the additive poset~$A$ is    plain, its underlying $\Z/2\Z$-vector space   is finite-dimensional. Therefore,    the equality $\cap_{s\in S} \Ker s=0$ implies that $S$ generates $A^*$.  By Lemma~\ref{lemma59}, there is a finite
  $n$-dimensional CW-complex $X$ and  an isomorphism     $  H^n(X; \Z/2\Z) \to A^*$   carrying   $S_X$ onto $S$. The dual isomorphism $A \to H_n(X; \Z/2\Z)$ is an isomorphism of the additive poset~$A$ onto the homology poset of~$X$.


\section{Homology posets of graphs}\label{Homological additive posets of graphs}

\subsection{Graphs and atoms} By a \emph{graph}, we mean a 1-dimensional CW-complex. A graph~$\Gamma$ is obtained by attaching  copies of the segment $[-1, 1]$ to a discrete set of points. The points of the latter set   are called vertices of~$\Gamma$ and the copies of $[-1, 1]$ in question are called edges of~$\Gamma$. Each edge~$e$ is attached to two (possibly, coinciding)  vertices  called the endpoints of~$e$.  A 1-cycle in~$\Gamma$  is a finite set  of edges such that every vertex of~$\Gamma$ is   incident to an even number of edges in this set (counting with multiplicities).  Every element of the group $H=H_1(\Gamma; \Z/2\Z)$ is represented by a unique 1-cycle.
As in Section \ref{Additive posets in topology}, we    give~$H$ the partial order
    $a\leq b$ whenever the 1-cycle representing $a\in H$ is  contained in the 1-cycle representing $b\in H$. As we know, this partial order is invariant under subdivisions of~$\Gamma$ and turns~$H$ into an additive poset,    the homology poset of $\Gamma$.

 \begin{lemma}\label{reprgraph45}   Let $\Gamma$ be a graph.
For any nonzero homology class  $a\in  H=H_1(\Gamma;\Z/2\Z)$, there is an embedded circle   in~$\Gamma$ representing a   homology class $b\in H$ such that $b\leq a$.     \end{lemma}

    \begin{proof}   Let  $E$ be a 1-cycle in~$\Gamma$ representing~$a$. If $E$ contains a loop,    that is  an edge with coinciding endpoints, then this loop  determines an embedding
      $S^1 \hookrightarrow   \Gamma$ representing a   nonzero homology class $b\in H $ such that  $b\leq  a$.
    Suppose   that~$E$ has no loops.   Pick any edge $e_1\in E$ with endpoints $x_1,x_2$.   Since $E$ is a 1-cycle, the vertex $x_2$ is incident to an edge $e_2\in E \setminus \{e_1\}$. Let $x_3$ be the   endpoint of~$e_2$ distinct from $x_2$ (possibly, $x_3=x_1$).  Since $E$ is a 1-cycle, the vertex $x_3$ is incident to an edge $e_3\in E \setminus \{ e_2\}$. Continuing by induction,   we construct   edges $e_1, e_2, e_3,... \in E$   such that any  two consecutive edges  $e_{i-1}, e_{i}$  are distinct and share a  vertex $x_i$. Since~$E$ is a finite set,  there must be   indices $k < l$ such that $x_k=x_l$. We take such  $k,l$ with the smallest  $l-k  $. Then the edges $e_k, e_{k+1},..., e_{l-1}$ are  pairwise distinct and have no common vertices except that  $e_{i-1}, e_{i}$    share the  vertex $x_i$ for all $i$ and $e_k$, $e_{l-1}$ share the   vertex $x_k=x_l$. Then the edges $e_k, e_{k+1},..., e_{l-1}$ form  an embedded circle     representing a    homology class $b\in H $ such that  $b \leq  a$.     \end{proof}

\begin{theor}\label{reprgraph}   Let $\Gamma$ be a graph.
A homology class  $a\in   H_1(\Gamma;\Z/2\Z)$ is represented by an embedded  circle   in~$\Gamma$ if and only if~$a$ is  an atom.     \end{theor}

    \begin{proof} If $a$ is represented by an embedded circle, then $a$ is an atom by Theorem~\ref{repr}.
    Conversely, if $a$ is an atom, then by Lemma \ref{reprgraph45}, there is a  homology class $b\in H_1(\Gamma;\Z/2\Z)$ such that $b\leq a$ and~$b$ is represented by an embedded circle. The latter condition implies that $b\neq 0$. Since~$a$ is an atom,   $b=a$, so that~$a$ is represented by an embedded circle.
    \end{proof}


\subsection{Geometric tiles}\label{Geometric tiles}    To extend Theorem \ref{reprgraph}  to  tiles,  we define  certain  graphs called geometric tiles.  Recall  that   a \emph{wedge} of    graphs $\Gamma_1 $ and $\Gamma_2$ is a   graph
    obtained from the disjoint union $\Gamma_1 \amalg \Gamma_2$ by identifying a vertex of~$\Gamma_1$ with a vertex of~$\Gamma_2$.  The wedge may depend on the choice of the vertices.


    A \emph{geometric tile  of weight
    1} is a graph  whose underlying topological space is homeomorphic to a circle. Suppose that   geometric tiles  of weights $1, 2, ..., n$ are already defined for an integer $n\geq 1$. A \emph{geometric tile  of weight
    $n+1$} is a graph which is either  a wedge of  a geometric tile  of weight~$n $ and a geometric tile  of weight~1 or  a disjoint union  of two geometric tiles  of weights $p,q \leq n$ with
    $p+q=n+1$.     For example, a disjoint union of $n\geq 1$ geometric tiles  of weight
    1 is a geometric tile  of weight
     $n$. An   induction on the  weight $w(T) $ of a   geometric tile~$T$ shows that~$T$ expands as a union of $w(T)$ embedded circles    meeting    only in  vertices of~$T$ (or not at all).
    Moreover, $T$ does not contain embedded circles other than those in this expansion (as is easily shown by induction on $w(T)$).


Given a graph $\Gamma$ and a homology class  $a\in   H_1(\Gamma; \Z/2\Z)$, the edges of $\Gamma$ belonging to the  unique 1-cycle representing~$a$ and their endpoints form a   graph     $\Gamma_a \subset \Gamma$  called the \emph{support} of $a$.

    \begin{theor}\label{tilegraph}   Let $\Gamma$ be a graph.
An element of   $  H=  H_1(\Gamma;\Z/2\Z)$ is a tile if and only if its support   is a geometric tile.     \end{theor}

    \begin{proof}    
If the support $\Gamma_a$ of $a\in   H $ is a geometric tile, then $a$ expands as a sum of $w(\Gamma_a) $ homology classes represented by embedded circles in $\Gamma_a$    meeting only  in the vertices. By Theorem \ref{reprgraph}, these homology classes are atoms. Clearly, they   are pairwise independent and lie in the tail $H_a$ of~$a$.  Moreover,  any atom in $H_a$ is represented by an embedded  circle  in $\Gamma_a$,  and so  is equal to one of the atoms in our expansion of~$a$. Thus,  all atoms in $H_a$ are pairwise independent, i.e.,   $a$ is a tile.

    Conversely, for any $a\in H $, its tail $H_a$   is a finite additive poset. If~$a$ is a tile, then, by Theorem~\ref{tilestiles},  all atoms $a_1,..., a_n$ of $ H_a$ are pairwise independent and their sum is equal to~$a$.   By Theorem \ref{reprgraph}, the atoms  $a_1,..., a_n$ are represented by embedded circles, respectively,   $S_1,..., S_n \subset \Gamma_a$. Since   $a_1,..., a_n$ are pairwise independent and $a_1+...+ a_n=a$, these~$n$ circles
  meet only in   vertices of~$\Gamma_a$ and their union is equal to $\Gamma_a$.
    We  prove that $\Gamma_a$ is a geometric tile   by induction on   $n$.  If $n=1$, then $a=a_1$ and $\Gamma_a=S_1$ is   a   circle.  We   explain the induction step.  
    Set $b=a_1+\cdots + a_{n-1}$.  Since $a_i \leq a$ for all $i$, we have $b  \leq a$. Since~$a$ is a tile, so is~$b$.   By the induction assumption, the support
   $\Gamma_b= S_1 \cup \cdots \cup S_{n-1}$ of~$b$  is a geometric tile.  We claim that the circle $S_n$  cannot meet a connected component of $\Gamma_b$ in more than one point. This    easily implies  that $\Gamma_a=\Gamma_b \cup S_n$ is a geometric tile. To prove the claim, suppose that $S_n$ meets a component $\Delta$ of~$\Gamma_b$ in  two or more   vertices.  Then there is   an embedded arc $\alpha\subset S_n$   meeting $\Delta$  precisely in its   endpoints. Pick an embedded path $\beta$ in $\Delta$ connecting  the  endpoints of $\alpha$. The union $\alpha \cup \beta$ is an embedded circle in $\Gamma_a$ distinct from  $S_1,..., S_n$. This is  impossible because such a circle represents an atom in $H_a$, and
  all atoms of $H_a$   belong to the list $a_1,..., a_n$. This proves the  claim  above and completes the proof of the theorem.
 \end{proof}


\subsection{Remarks} 1. If a graph $\Gamma$ is contained in a graph $\Gamma'$, then the inclusion homomorphism $   H_1(\Gamma;\Z/2\Z)
\to  H_1(\Gamma';\Z/2\Z)$   is an embedding of additive posets   carrying atoms to atoms and tiles to tiles. A similar claim holds for   CW-complexes.

2. Given a graph $\Gamma$ and atoms   $ b,c \in  H
 = H_1(\Gamma;\Z/2\Z)$, one may ask whether  the circles  in~$\Gamma$ representing  $b$ and $c$      are    disjoint.  A necessary condition is that  $b$, $c$ are independent and $b+c$ is a tile. However, this condition is insufficient because it  leaves open the possibility that the circles representing~$b$ and~$c$  meet in a vertex of~$\Gamma$.
 Consider the  map  $\chi: H\to \Z $ carrying   a  homology class to the   Euler characteristic   of its support.  
 The circles  in~$\Gamma$ representing  the atoms $ b,c \in  H
 $    are   disjoint     if and only if  $b,c$ are independent, $b+c$ is a tile,   and $\chi(b+c)=0$.
  A   homology class $a \in H$
   is represented by  a     disjoint union of embedded circles   if and only if~$a$ is a tile and   any   two atoms  $\leq a$   are represented by disjoint   circles.


 3.  A graph is  \emph{finite} if  its  sets  of vertices and   edges are finite. All geometric tiles are finite graphs and the support of any homology class of a graph is a finite graph.  The  homology poset of a finite graph is   plain and its
  complexity   is smaller than or equal to the number of edges of the graph.
By Section~\ref{Atoms as generators},  all elements of this poset expand as sums of atoms.

    \section{Realization   by   graphs}

  An additive poset   is \emph{realized}  by a graph~$\Gamma$  if it is isomorphic to the homology poset of~$\Gamma$. For instance, for a  set~$I$, the additive poset $2^I_f$ is realized by a wedge of circles, cf. Example~\ref{exewed}.1.

  An additive poset   is \emph{realizable}   if it is realized by a certain graph.


 \begin{lemma}\label{realsub}  If $A $ is a realizable additive poset, then for any $a\in A$, the additive subposets $A_a$ and $A^a$ of~$A$ are realizable.   \end{lemma}

\begin{proof} Since~$A$ is realizable, we can identify~$A$ with the homology poset $H=H_1(\Gamma; \Z/2\Z)$ of a graph~$\Gamma$. Pick any $a\in H$ and consider the support $\Gamma_a\subset \Gamma$  of~$a$ (see Section~\ref{Geometric tiles}). Observe that the inclusion homomorphism $H_1(\Gamma_a; \Z/2\Z) \to H  $ is injective and   its image  is the  group $H_a=\{b\in H \,\vert \, b\leq a\}$. Since the partial order in~$H$ is induced by the inclusion of cycles in~$\Gamma$,  the induced partial order in $H_a$ is induced by the inclusion of cycles in~$\Gamma_a$. Therefore the additive poset $H_a$ is realized by $\Gamma_a$. Similarly,  the edges of $\Gamma$ not belonging to  $\Gamma_a$ and their endpoints form a   graph     $\Delta_a \subset \Gamma$.
The inclusion homomorphism $H_1(\Delta_a; \Z/2\Z) \to H $ is injective and   its image is the group  $H^a \subset H $. Indeed, the support of a   homology class $
b\in H$ is contained  in $\Delta_a$ if and only if $a\leq a+b$.
Therefore the additive poset $H^a$ is realized by $\Delta_a$.
\end{proof}

  The following theorem yields examples of non-realizable additive posets.

 \begin{theor}\label{trealgraphs}  Let $m $ be a positive integer and  let   $A_m$ be an $m$-dimensional $\Z/2\Z$-vector space   with  trivial partial order (see Section~\ref{exam16}.1).  The additive poset $A_m$ is realizable for $m\leq 4$ and non-realizable for   $m\geq 5$.  \end{theor}

\begin{proof}  Observe that:  $A_1$ is realized by the graph formed by one vertex and one edge (which is a loop);  $A_2$ is realized by the graph  formed by   two  vertices and three  connecting them edges;  $A_3$ is realized by the   complete graph on four vertices;  $A_4$ is realized by the     graph   $K_{3,3}$  having  six vertices three of which are connected to each of the other three by a single edge. In each of these cases,  all nonempty 1-cycles are embedded circles. Therefore the partial order in the homology poset is trivial.

It remains to show that there are no    graphs $\Gamma$ with trivial  partial order in  $H=H_1(\Gamma; \Z/2\Z)$    and with
$\dim(H)\geq 5$. Suppose that there is such a~$\Gamma$.
Pick a  nonzero  $a\in H$. Since the partial order in~$H$ is trivial, $a$ is an atom and $H^a=0$. By Theorem \ref{reprgraph}, the support $\Gamma_a$ of~$a$ is  an embedded circle in~$  \Gamma$.   Defining the graph $\Delta=\Delta_a\subset \Gamma$ as in the proof of Lemma~\ref{realsub}, we obtain  that $H_1(\Delta; \Z/2\Z)=H^a=0$.    So,    all connected components of~$\Delta$ are trees.
Let $\{T_i\}_i$   be  the components of $\Delta$ meeting $\Gamma_a$ (in some vertices).    Clearly,
$$\dim (H)=1+ \sum_{i}  (\vert \Gamma_a \cap T_i \vert-1) $$
where    $\vert \Gamma_a \cap T_i \vert\geq 1$ is the number of common vertices of $\Gamma_a$ and $T_i$.
Since $\dim(H) \geq 5$, at least   one of the following three conditions hold:

(a) the circle  $\Gamma_a$ meets a component $T$ of $\Delta$   in $\geq 3$  vertices and meets  another  component~$T'$ of~$\Delta$   in $\geq 2$ vertices;

(b) the circle $\Gamma_a$ meets a component $T$ of $\Delta$  in $\geq 5$  vertices;

(c)   $\Gamma_a$ meets four distinct components  $T_1, T_2, T_3, T_4$ of $\Delta$   in   two vertices each.

  In every case, we will construct two non-empty 1-cycles in $\Gamma_a\cup \Delta$     meeting  only in vertices. The  homology classes $b,c\in H$ of these 1-cycles are nonzero and satisfy $b\leq b+c$. This  contradicts the assumption that the partial order in $H$ is trivial.

For  any  distinct points $u,v$ of a tree, we   denote by $  \overline{uv}$ an embedded path
 connecting $u,v$ in this tree.
 Fix from now on an orientation in the circle $\Gamma_a$. For any distinct  points $u,v\in \Gamma_a$ we let $\widehat{uv} $ be the   arc  on $\Gamma_a$  leading from~$u$ to~$v$. If     $u,v$  also lie on the same component of~$\Delta$, then  we have the embedded circle
 $$S_{uv}= \widehat{uv} \cup \overline{uv} \subset   \Gamma_a\cup \Delta \subset \Gamma.$$

Case (a). Let $u,v,w  $ be   distinct points of $\Gamma_a \cap T$ and let $x,y$ be   distinct points of $\Gamma_a \cap T'$. The points $u,v,w $ split   $\Gamma_a$ into three arcs, and at least one of them, say, $\widehat{uv}$   contains neither~$x$ nor~$y$. 
Consider the   arc   $\alpha \subset \Gamma_a \setminus \widehat{uv}$  with the endpoints  $x,y$. Then   the following embedded  circles are disjoint:  $$S_{uv} = \widehat{uv} \cup \overline{uv} \subset \Gamma_a \cup T    \quad {\text{and}} \quad \alpha \cup \overline{xy}\subset \Gamma_a \cup T'  . $$

Case (b). Let $u,v,w,x,y$ be   distinct points of $\Gamma_a \cap T$ enumerated in the  cyclic order on   $\Gamma_a$.    If the paths $\overline{uv}$ and $\overline{xy}$ in the tree~$T$ meet only in vertices, then so do the   circles $S_{uv},  S_{xy}  $ and we are done.  Suppose that the paths $\overline{uv}$ and $\overline{xy}$ have a  common edge. The complement of the interior of this edge in~$T$ is a union of two disjoint   subtrees  of~$T$. Clearly,   the points $u, v$  lie in different subtrees and so do $x,y$. If $v,x$ lie in the same subtree, then  $y,u$ lie in the other subtree and    $ S_{vx} \cap S_{yu} =\emptyset$. Suppose  that $v,y$ lie in the same subtree and $u,x$ lie in the other subtree.
If~$w$ lies in the same subtree as~$v $ and~$y$, then  $S_{vw}\cap S_{xu}=\emptyset$.
If~$w$ lies in the same subtree as~$u$ and~$ x$, then  $S_{wx}\cap S_{yv}=\emptyset$.

Case (c).
For $i=1,..., 4$, let $u_i, v_i$ be the points of $ \Gamma_a \cap T_i $.  If there are disjoint arcs $\alpha, \beta$ on the circle $ \Gamma_a$ such that $\partial \alpha=\{u_i, v_i\}$ and $\partial \beta=\{u_j, v_j\}$ for some distinct $i, j\in \{1,2,3,4\}$, then
 the following embedded circles are disjoint:  $$\alpha \cup \overline{u_i v_i} \subset \Gamma_a \cup T_i   \quad {\text{and}} \quad \beta \cup \overline{u_j v_j} \subset \Gamma_a \cup T_j  . $$
 If there are no such  $\alpha, \beta$, then up to  enumeration of
 $T_1, T_2, T_3, T_4$ and   permutation of the symbols $u_i, v_i$, the cyclic order of the points $\{u_i, v_i\}_{i=1}^4$ on  $\Gamma_a$ is
 $$u_1, u_2, u_3, u_4, v_1, v_2, v_3, v_4. $$
 Then   the following two unions of arcs represent disjoint 1-cycles:  $$\widehat{u_1u_2} \cup \overline{u_2v_2} \cup \widehat{v_2v_1} \cup \overline{v_1 u_1}  \subset \Gamma_a \cup T_1 \cup T_2   $$
 and
  $$\widehat{u_3u_4} \cup \overline{u_4v_4} \cup \widehat{v_4v_3} \cup \overline{v_3 u_3}  \subset \Gamma_a \cup T_3 \cup T_4  .$$

 \end{proof}

 \begin{corol}\label{trealgcprphs}   Let $A $ be a realizable additive poset. Then for any $a\in A$, the partial order in    $A^a$ is non-trivial or $\dim  A^a \leq 4$. \end{corol}

 As an application, we deduce that the additive poset $A=2^I\oplus B$ is not realizable for any   set~$I$ and any $\ZZ/2\ZZ$-vector space~$B$ of   dimension $\geq 5$ carrying trivial partial order. Indeed, the set~$I$ represents an element $a\in 2^I\subset A$ such that $A^a=B$.

\section{Open questions}

1. Are all  finite additive posets      plain?
In view of Theorem \ref{simple}, this question may be restated as follows: is it true that for any finite additive poset~$A$, the   set of  order-preserving linear functionals on~$A$ is  separating?

2.  A \emph{rank}  of an additive poset $A$ is a map $r:A\to \{0,1,2, ... \}$ such that $r^{-1}(0)=0$,  for any independent $a, b\in A$ we have $r(a+b)=r(a) + r(  b)$  and for any  finite nonempty set $K\subset A$, the integer
 $$\sum_{\emptyset \neq J\subset K} (-1)^{\vert J \vert +1} \,r (\sum_{a\in J} a) $$ is nonnegative and divisible by $2^{\vert K \vert -1} $. An   injective morphism of additive posets   $\varphi : A \to 2^I_f$, where $I$ is a set, determines a rank  of~$A$ by $r(a)=\vert \varphi(a) \vert$ for all $a\in A$. Thus,
 every plain  additive poset has a rank.
Does every  finite additive poset have a rank?  Are all finite additive posets  having a rank  plain?
Is every rank of an additive poset   induced by an injective morphism  into an additive powerset?

3.  The additive poset $A_m$ from Theorem~\ref{trealgraphs} is plain, cf.  Example~\ref{exam22}.1. It is easy to see that  $c(A_m)\leq m (m+1)/2$ for all~$m$ and  $c(A_m)= m (m+1)/2$ for $m \leq 3$. Is the latter equality true  for all~$m$?

4.    If $A$ and~$B$ are plain additive posets, then so is their direct sum $ A\oplus  B$ and $c(A \oplus B)\leq c(A)+c(B)$.  Is this inequality    an equality?

5. For  any $m\geq 1$,   an $m$-antichain in an additive poset $A$ is an antichain in~$A$ whose elements generate a vector subspace of~$A$ of dimension $\leq m$. The  $m$-width  $w_m(A)$ of~$A $   is the maximal number of elements in an $m$-antichain in $A$.  Clearly, $$1=w_1(A) \leq w_2(A) \leq \cdots \leq w_{\dim(A)}(A)=w(A).$$
Compute the $m$-width for all finite additive powersets and all~$m$.


6.  Describe algebraically the class of   posets realizable by  graphs.

%

                     \end{document}